\documentclass[11pt,reqno]{amsart}
\usepackage{amsfonts,amsmath,amssymb,epsf,color,array,colortbl,amscd,bbm,tensor,tikz,comment,caption,soul, enumitem}
\usepackage{tikz}
\usetikzlibrary{shadings}
\usetikzlibrary{matrix,arrows}
\usepackage{fullpage}
\theoremstyle{plain}
\newtheorem{theorem}{Theorem}
\newtheorem{proposition}[theorem]{Proposition}

\newtheorem{corollary}[theorem]{Corollary}
\newtheorem{lemma}[theorem]{Lemma}
\theoremstyle{definition}
\newtheorem{definition}[theorem]{Definition}
\newtheorem{remark}[theorem]{Remark}

\allowdisplaybreaks

\newcommand {\Z}{\mathbb{Z}}
\newcommand {\Q}{\mathbb{Q}}
\newcommand {\R}{\mathbb{R}}
\newcommand {\C}{\mathbb{C}}

\newcommand{\be}{\begin{equation}}
\newcommand{\ee}{\end{equation}}

\newcommand{\sgn}{\text{sgn}}
\newcommand{\re}{\text{Re}}

\begin{document}
\title{\bf  Higher Rank Partial and False Theta Functions and Representation Theory}

\author{Thomas Creutzig and Antun Milas }
\thanks{T.C. was supported by an NSERC Research Grant (RES0020460)}
\thanks{ A.M. was supported by a Simons Foundation Collaboration Grant ($\#$ 317908) and support from the Max Planck Institute for Mathematics, Bonn.}
\maketitle

\begin{abstract} We study higher rank Jacobi partial and false theta functions (generalizations of the classical partial and false theta functions) associated to positive definite rational lattices. 
In particular, we focus our attention on certain {\em Kostant's} partial theta functions coming from  ADE root lattices, which are then linked to representation theory of W-algebras.
We derive modular transformation properties of {\em regularized}  higher rank partial and false theta functions as well as Kostant's version of these. Modulo conjectures in representation theory, as an application, we compute regularized quantum dimensions of atypical and typical modules of "narrow" logarithmic $W$-algebras associated to rescaled root lattices. This paper substantially generalize our previous work 
\cite{CM1} pertaining to $(1,p)$-singlet $W$-algebras (the $\frak{sl}_2$ case). Results in this paper are very general and are applicable in a variety of situations. 
\end{abstract}

\section{Introduction} 

\subsection{Motivation}

Theta functions have been enormously useful in many different branches of mathematics.
Not only that they appear in number theory and algebraic geometry, but also in representation theory, combinatorics, topology and theoretical physics.
More recently, certain incomplete theta series, called partial theta functions, as well as theta functions with "wrong signs", called false theta functions, have also appeared in various modern aspects of these areas. For example, in \cite{GL}, certain limits (called $k$-limits) of colored Jones functions for alternating links where shown to be higher rank partial theta-like series. Also, generating functions for colored vector partitions are given by the Fourier coefficients of certain mutli-variable meromorphic Jacobi forms for negative definite quadratic forms. The latter being again (if known) higher rank partial theta series \cite{BCR, BRZ}. The most prominent example is arguably the crank partition \cite{AG}. We also point out {\em conical} theta 
functions recently studied in \cite{FKR}. Of course, "rank one" partial and false theta had appeared much earlier starting in the works of Rogers and of Ramanujan in connection to $q$-hypergeometric identities (see \cite{AB,Wa}). For $q$-series identities coming from a product of {\em two} partial theta functions see \cite{AW}.

%In this paper we consider analytic properties of these theta's "ugly" cousins. By extending our previous works \cite{CM1,CMW} to "higher ranks", we show that 
%these functions also have an interesting life on their own. 

Let us first define the most general partial theta function used in the paper.
Denote by $L=\Z\alpha_1 \oplus \dots \oplus \Z\alpha_n$ a rank $n$ lattice with positive definite quadratic form $\left< \ \ , \ \ \right> : L\times L \rightarrow \mathbb Q$, that is, the Gram matrix $A=(\langle \alpha_i,\alpha_j \rangle )_{i,j=1}^n $ is positive definite. We also let $|| u ||=\sqrt{\langle v,v \rangle }$. It will be convenient to use the usual bilinear form on $\mathbb{C}^n$, denoted by $( \cdot , \cdot)$. If  $v,w \in L$ and we view $v,w$ as vectors expressed in $\mathbb{C}^n$ in the $\alpha$-basis, we have $(Av,w)=(v,Aw)=\langle v,w \rangle$.

%{\bf How about we use $\langle , \rangle $ for this positive definite form and $( , )$ for the usual bilinear form on $\mathbb{C}^n$? I'm getting confused  which one is where. OK let's do that.}

Define the positive cone with respect to the basis $\alpha_1, \dots, \alpha_n$ as $L^+ := \Z_{>0}\alpha_1 \oplus \dots \oplus \Z_{>0}\alpha_n$, then the Jacobi partial theta function of $L^+$ is
defined by
\begin{equation} \label{P-def}
P(u, \tau) = \sum_{\lambda\in L^+ } q^{\frac{1}{2} \left<\lambda, \lambda\right>} e^{2\pi i \left<u, \lambda\right>},
\end{equation}
where $q=e^{2 \pi i \tau}$, $\tau$ in the upper half-plane, and $u \in L \otimes_{\mathbb{Z}} \mathbb{C} \cong \mathbb{C}^n$. We will be primarily concerned with a certain regularization of  
(\ref{P-def}), denoted by $P_{\epsilon}(u,\tau)$, as well as various specializations of $P_\epsilon$.
It is sometimes important to take an alternating sum of several partial theta functions $P(u,\tau)$, which can lead to Jacobi {\em false} theta functions. A typical example is the Jacobi  false theta function \cite{CM1,CMW, BFM}, which, when specialized, gives the classical Rogers' false theta function.
In the special case when $L$ is a rescaled root lattice of a simply laced Lie algebra, we will also consider what we end up calling Kostant's partial theta function  $K(u,\tau)$ - a partial theta function
weighted with the value of the Kostant partition function $K(\beta)$, $\beta \in Q_+$, obtained by expanding the Weyl denominator. The corresponding 
regularization will be denoted by $K_\epsilon(u,\tau)$.
In addition to those, we will also consider regularized false versions of $K_\epsilon$ obtained by averaging over the finite Weyl group $W$.
%We should say that there is almost nothing about these and similar objects beyond the rank one case. But we stress again that specialized partial theta series (where $u=\epsilon=0$)
%do seem to appear in the literature. 

Although one can argue that (regularized) partial and false theta functions deserve to be studied {\em per se}, for us, the main motivation comes from 
their interplay with irrational vertex operator algebras (VOA) and related non semi-simple representation categories. As it is customary in this area of mathematics, any modular-like function
is expected to be found in the realm of characters of vertex algebra modules. In this paper we mostly care about certain "narrow" vertex subalgebras of lattice vertex algebras associated to 
rescaled root lattices. More examples of vertex algebras with module characters being higher rank false theta series can probably be obtained via Heisenberg cosets of affine VOAs and W-algebras of negative admissible level, see e.g. \cite{CRW, ACR}. But apart from connection with the character theory there is also a
strong link to tensor categories. Recall that characters of $C_2$-cofinite, simple  VOAs of CFT-type that are its own contragredient dual are modular \cite{Mi} and if the VOA is in addition rational its representation category is also modular \cite{H}. It is an important and very open question to understand the representation categories of VOAs that are neither semi-simple nor have only finitely many simple objects. A natural expectation is that under some niceness conditions (like $C_1$-cofinitness \cite{CMR}) this representation category might be ribbon, or at least a subcategory of interest is. While in rational VOAs the Hopf links coincide with normalized modular $S$-matrix entries of torus one-point functions it is not clear if there is such an analog for VOAs with infinitely many objects. Our previous work on rank one false theta functions and so-called singlet W-algebras suggests strongly that quantum dimensions of characters coincide with modified "logarithmic" Hopf links, see Theorem 1 of \cite{CMR} but also \cite{CG,CG2} on similar statements for $C_2$-cofinite but non-rational VOAs. Unfortunately, not much is known about the tensor category of the singlet VOAs and their higher rank analogs. For example only for the simplest rank one example structure of the tensor ring of finite length modules is known \cite{AM3, CMR}. For works on the representations of these VOAs see also \cite{AM1, AM}. We hope that results from this paper will lead to better understanding of the structure of logarithmic Hopf link 
invariants in "higher rank" theories.

\subsection{Content and main results}

The paper is organized as follows. Sections 2 and 3 deal with partial and false theta functions and can be read without any familiarity with vertex algebra theory and very little or no knowledge of representation theory of Lie algebras. Any reader with basic knowledge of theta functions and analytic number theory should be able to follow. Although in sections 4 and 5 we are using vertex algebra theory, even this part should be fairly readable as we focus primarily on explicit characters of vertex algebra modules. The only result that is purely of interest to the practitioners  of vertex algebras is Theorem \ref{non-generic}. But this result has no bearing on the rest of the paper, so it can be omitted.

In more detail, in Section 2, we are concerned with modular transformation properties of the regularized partial theta function $P_\epsilon(u,\tau)$ under the Jacobi group $SL(2,\mathbb{Z}) \ltimes \mathbb{Z}^{2n}$. Behavior under translations is quite clear and so we need to mainly discuss the $S$-transformation, which is also most interesting. In parallel to \cite{CM1}, we first focus on the region ${\rm Re}(\epsilon_i)<0$ for all $i=1,...,n$, which is easier to study due to convergency property of its asymptotic expansion. Our first result is given in Theorem \ref{first}, where we gave a closed formula for $P_\epsilon(\frac{u}{\tau},-\frac{1}{\tau})$ in this region. Then we embark to extend this formula beyond ${\rm Re}(\epsilon_i) <0$ and 
away from ${\rm Re}(\epsilon_i)=0$. The main result in this direction is Theorem \ref{S-partial} stating that, under the condition $ \epsilon_i \notin (i \mathbb{R})$, we have
\begin{equation}
P_\epsilon\left(\frac{u}{\tau}, -\frac{1}{\tau}\right)=h_\epsilon(u, \tau)+e^{\frac{\pi i}{\tau}\left(u, Au\right)} (2i)^{-n} \sqrt{\frac{(-i\tau)^n}{ \det A}} I.
\end{equation}
where  $h_\epsilon$  is defined in (\ref{h-function}) and  $I$ is a certain sum of $k$-fold integral defined on p.7.

In Section 3, we switch our attention to regularized Kostant's partial theta functions. For the $S$-transformation,   
our answer is given in terms of a contour integral where the contour depends on the region of the regularization parameter. We evaluate this integral for an interesting region in Section 3.1. 
The main result here is 
\begin{equation}\label{K-formula-intro}
K_\epsilon\left(\frac{u}{\tau}, -\frac{1}{\tau}\right) = \frac{e^{\frac{\pi i}{\tau}\left(u, Au\right)}}{(2i)^{|\Delta_+|}} \sqrt{\frac{(-i\tau)^n}{ \det A}}
\sum_{R\in P_n} X_R,
\end{equation}
where $X_R$ are defined on p.10 and $P_n$ denote the power set of $\{1,2,...,n \}$.

At this point, we change the pace and 
Section 4 then applies our findings to compute asymptotic dimensions (also called quantum dimensions) of higher rank logarithmic "narrow" W-algebras, generalizing 
the (1,p)-singlet $W$-algebra studied in \cite{CM1}. These higher rank $W$-algebras are introduced in \cite{M1,BM1} and will be denoted by $W^0(p)_Q$, where $p \in \mathbb{N}_{\geq 2}$ and $Q$ is a root lattice of type ADE. 
These W-algebras are associated to rescaled root lattices of simply laced Lie algebras $\mathfrak{g}$ and (modulo conjectures in representation theory) their atypical characters are expressed in terms of alternating sums of Kostant partial theta series. Explicitly, if $\mu$ denotes an element in $L^0$ (the dual lattice of $\sqrt{p} Q$), then the regularized character
\begin{equation}\nonumber
\begin{split}
{\rm ch}[W^0(p,\mu)_Q]^\epsilon(\tau) &:= \frac{(-1)^{|\Delta_+|}}{\eta(\tau)^{rank(Q)}} \sum_{w \in W} (-1)^{\ell(w)} A_w^\epsilon(\tau) \\
A_w^\epsilon(\tau)&:= e^{2\pi(v+\frac{1}{2}(1, \dots, 1), \epsilon)}  q^{\frac{1}{2}\left(v, Av\right)}   K_{\epsilon}(v \tau,\tau),
\end{split}
\end{equation}
where $W$ is the Weyl group of $\frak{g}$ and $v \in \mathbb{C}^n$ depends on $\mu$ and $w \in W$. It does not require much effort to compute now the $S$-transformation of
${\rm ch}[W^0(p,\mu)_Q]^\epsilon(\tau)$ in the ${\rm Re}(\epsilon_i)<0$ region (see Corollary \ref{gauss-char}). 
Equipped with this formula we can now compute quantum dimensions of modules (also inside ${\rm Re}(\epsilon_i)<0$). It turns out that the asymptotic  (or quantum) dimensions are expressible using characters of simple $\mathfrak{g}$-modules.
$${\rm qdim}[W^0(p,\mu)_Q]^\epsilon = e^{-2\pi (\gamma+\hat{\lambda}, \epsilon)} \chi_{- \sqrt{p} \bar{\lambda}}\left(\frac{-i\epsilon}{p}\right),$$
where $\chi_\beta$ denotes the character of the finite-dimensional irreducible $\frak{g}$-module $V(\beta)$ and $\bar{\lambda}$, $\hat{\lambda}$ depend on $\mu$.
As already noticed in the rank one case \cite{CMW,CM1,BFM}, outside the ${\rm Re}(\epsilon_i)<0$ region, regularized asymptotic dimensions exhibit interesting properties. 
For higher ranks, it seems very delicate to compute their behaviors throughout the whole $\epsilon$-space. Still, as an illustration, we compute regularized dimensions inside a cell 
when ${\rm   Re}(\epsilon_i)$ is large (this way we omit inconvenient Stokes hyprplanes appearing close to imaginary axes). Interestingly, but not completely unexpectedly \cite{CMW},
the resulting quantum dimensions are closely related to $S$-matrices of WZW models.  Our result says that for certain $k \in \mathbb{Z}^n$ inside ${\rm Re}(\epsilon_i)>0$,  
subject to conditions (i) and (ii), quantum dimensions are (see Theorem \ref{quantum-dim-pos}):  
$${\rm qdim}[W^0(p,\mu)_Q]^\epsilon = e^{-2\pi (\gamma+\hat{\lambda}, k)} \chi_{-\sqrt{p}\bar\lambda}\left(-\frac{k}{p}\right).$$
Quantum dimensions of  typical representations are also computed there.

\subsection{Future work}
Our future work will extend the computation in Section 6 to other subregions in ${\rm Re}(\epsilon)>0$, starting with  $Q=A_2$ (the $\frak{sl}_3$ case). 
In the rank one case, regularized asymptotic dimensions coincide with logarithmic Hopf link invariants of the restricted unrolled quantum group of $\mathfrak{sl}_2$ at appropriate root of unity \cite{CMR}. Another future goal is to relate higher rank unrolled quantum groups to the higher rank "narrow" W-algebras. 

{\bf Acknowledgements:} A.M. thanks T. Arakawa for a useful correspondence about affine $W$-algebras and Theorem 10.

\section{Modular transformations of regularized partial theta functions}

Let $x=(x_1, \dots, x_n)$ and $y=(y_1, \dots, y_n)$ be vectors in $\C^n$. Denote by $( \ \ , \ \ )$ the standard $\mathbb{C}$-bilinear form $(x, y):= x_1 {y_1} + \dots + x_n {y_n}$. 
 \begin{definition}
Let $A$ be a symmetric positive invertible matrix (over $\Q$) of rank $n$ with entries in $\Z$.
Let $\epsilon \in \left(\C \setminus i\R\right)^n$ and $u \in \C^n$, then the regularized (or Jacobi) partial theta function of rank $n$ for the lattice $L=\left(\Z^n, A\right)$ with quadratic form defined by $A$ is
\[
P_\epsilon(u, \tau) := \sum_{k \in  \left( \Z_{\geq0}+\frac{1}{2}\right)^n} q^{\frac{1}{2} \left(k, Ak\right)} e^{2\pi i \left(k, Au\right)}e^{2\pi \left(k, \epsilon\right)}.
\]
\end{definition}
Note, that we can allways recover the unregularized partial theta function via
\[
P(u, \tau) := \sum_{k \in  \left( \Z_{\geq0}+\frac{1}{2}\right)^n} q^{\frac{1}{2} \left(k, Ak \right)} e^{2\pi i \left(k, Au\right)} = P_\epsilon\left(u+iA^{-1}\epsilon, \tau\right).
\]
First, it is easy to derive the properties of $P_\epsilon(u, \tau)$ under the action of translations  $\Z^n\tau+\Z^n$; this formula will not be used in the rest of the paper though.
\[
P_\epsilon(u+m \tau +\ell, \tau) = e^{\pi i \left(e, A\ell\right)} \sum_{k \in  \left( \Z_{\geq0}+\frac{1}{2}\right)^n} q^{\frac{1}{2} \left(k, A(k+2m) \right)} e^{2\pi i \left(k, Au\right)}e^{2\pi \left(k, \epsilon\right)}.
\]
Here $m, \ell \in \Z^n$ and $e=(1, 1, \dots, 1)\in\Z^n$.
%\[
%q^{-\frac{1}{2}\left(m, Am\right)} e^{-2\pi i \left(m, Au\right)} e^{-2\pi (m, \epsilon)} P_\epsilon(u, \tau). 
%\]

%{\bf AM: What exactly did we have in mind here? This formula is much more complicated. But I don't think we need it.}

%\[
%P_\epsilon(u+m\tau +\ell, \tau) = e^{\pi i \left(e, A\ell\right)} q^{-\frac{1}{2}\left(m, Am\right)} e^{-2\pi i \left(m, Au\right)} e^{-2\pi (m, \epsilon)} P_%\epsilon(u, \tau). 
%\]
%Here $m, \ell \in \Z^n$ and $e=(1, 1, \dots, 1)\in\Z^n$. 

We want to study the modular properties of regularized partial theta functions. They involve the following continuous part:
\begin{equation} \label{h-function}
h_\epsilon(u, \tau) := e^{\frac{\pi i}{\tau}\left(u, Au\right)} (2i)^{-n} \sqrt{\frac{(-i\tau)^n}{ \det A}} \int_{\R^n} \frac{q^{\frac{1}{2}\left(w, A^{-1}w\right)} e^{-2\pi i (u,w)} }{\prod\limits_{j=1}^n \left( \sin\left(\pi\left(w_j-i\epsilon_j\right)\right)\right)} d^nw.
\end{equation}
The first task is to compare the elliptic properties of this integral with the ones of the $S$-transformed partial theta function. The generalized Gauss integral is needed:
\begin{lemma}
Let $M$ be a symmetric positive-definite matrix of rank $n$ with real coefficients, and let $b \in \C^n$, then
\[
\int_{\R^n} e^{-\frac{1}{2}\left(w, M w\right)+\left(b, w\right)} d^nw = \sqrt{\frac{\left(2\pi\right)^n}{\det M}} e^{\frac{1}{2}\left( b, M^{-1} b\right)}.
\]
\end{lemma}
\begin{lemma}\label{lemma:translating}
Let $\alpha \in \mathbb C^n$ such that $\epsilon-\alpha\in \left(\C \setminus i\R\right)^n$ 
and let $u'=u-iA^{-1}\alpha\tau$. Then
\[
P_\epsilon\left(\frac{u}{\tau}, -\frac{1}{\tau}\right)= P_{\epsilon-\alpha}\left(\frac{u'}{\tau}, -\frac{1}{\tau}\right),\
\]
and 
\[
h_{\epsilon-\alpha}(u', \tau) = h_{\epsilon}(u, \tau) +
 e^{\frac{\pi i}{\tau}\left(u, Au\right)} (2i)^{-n} \sqrt{\frac{(-i\tau)^n}{ \det A}} \oint_{C_\alpha}   \frac{q^{\frac{1}{2}\left(w, A^{-1}w\right)} e^{-2\pi i (u,w)} }{\prod\limits_{j=1}^n \left( \sin\left(\pi\left(w_j-i\epsilon_j\right)\right)\right)} d^nw,
\]
where the contour $C_\alpha$ connects the regions $\R^n$ (clockwise) and $\R^n+i\alpha$ (counter clockwise) at infinity.  
For $\alpha \in i \mathbb{R}^n $ the contour integral vanishes. 
\end{lemma}
\begin{proof}
Looking at the definition of the partial theta function, we see that
\[
P_\epsilon(u+iA^{-1}\alpha, \tau)=P_{\epsilon-\alpha}(u, \tau).
\]
Hence, we can rewrite
\begin{equation}\nonumber
\begin{split}
P_\epsilon\left(\frac{u}{\tau}, -\frac{1}{\tau}\right)&= 
P_\epsilon\left(\frac{u-iA^{-1}\alpha\tau+iA^{-1}\alpha\tau}{\tau}, -\frac{1}{\tau}\right)= 
P_{\epsilon}\left(\frac{u'}{\tau}+iA^{-1}\alpha, -\frac{1}{\tau}\right)\\
&=P_{\epsilon-\alpha}\left(\frac{u'}{\tau}, -\frac{1}{\tau}\right). 
\end{split}
\end{equation}
The second identity is true, as 
\[
h_{\epsilon-\alpha}(u', \tau) = 
 e^{\frac{\pi i}{\tau}\left(u, Au\right)} (2i)^{-n} \sqrt{\frac{(-i\tau)^n}{ \det A}} \int_{\R^n+i \alpha}   \frac{q^{\frac{1}{2}\left(w, A^{-1}w\right)} e^{-2\pi i (u,w)} }{\prod\limits_{j=1}^n \left( \sin\left(\pi\left(w_j-i\epsilon_j\right)\right)\right)} d^nw,
\]
which can be seen by changing variable $w\mapsto w+i\alpha$; and since $\tau$ is in the upper half plane the integrand vanishes exponentially at infinity.
\end{proof}
In other words, we can reduce the problem to the case where all components of $\epsilon$ have negative real part. The next result is  
taken from \cite{CM1} (also see \cite{CMW})  and it follows directly from the residue theorem.
 
\begin{lemma} \label{one-dim-int} For $\mu \in \mathbb{C}$, let $C_{\mu}$ be a contour connecting $\mathbb{R}$ and $\mathbb{R}+ i\mu$ at infinity, such that the orientation 
of $\mathbb{R}$  is standard from $-\infty$ to $+\infty$. Then
\begin{equation*}
\begin{split}
\int_{\mathcal C_\mu} \frac{q^{x^2/2}z^{x}}{\sin (\pi(  x-i\epsilon))}\, dx\, &=\,  \left( {2 i }  \sum_{n\in\mathbb Z}(-1)^nz^{n+i\epsilon}q^{(n+i\epsilon)^2/2} \right) \delta(\epsilon,\mu)
%\int_{\mathcal C_\mu} \frac{q^{x^2/2}z^{x}}{\sin (\pi(  x-i\epsilon))^{m+1}}\, dx\, &=\,  \left( \frac{2 i }{\pi^{m} m !}  \sum_{n\in\mathbb Z}(-1)^n \frac{\partial^m}{\partial x^m} \left( z^{x}%q^{x^2/2}\right)_{x=n+i \epsilon} \right) \delta(\epsilon,\mu)
\end{split}
\end{equation*}
where
$$\delta(\epsilon,\mu)=\begin{cases}
0 &\mathrm{if}\ \mathrm{Re}(\epsilon)>0\ , \ \mathrm{Re}(\mu)< \mathrm{Re}(\epsilon)\\
1 &\mathrm{if}\ \mathrm{Re}(\epsilon)>0\ , \ \mathrm{Re}(\mu)> \mathrm{Re}(\epsilon) \\
-1 &\mathrm{if}\ \mathrm{Re}(\epsilon)<0\ , \ \mathrm{Re}(\mu)<\mathrm{Re}(\epsilon)\\
0. &\mathrm{if}\ \mathrm{Re}(\epsilon)<0\ , \      \mathrm{Re}(\mu)>\mathrm{Re}(\epsilon) \\
\end{cases}
$$
In particular, for ${\rm Re}(\epsilon-\mu)<0$ we have 
$$\int_{\mathcal C_\mu} \frac{q^{x^2/2}z^{x}}{\sin (\pi(  x-i\epsilon))}\, dx\,=i (1+{\rm sgn}({\rm Re}(\epsilon))) \left(\sum_{n\in\mathbb Z}(-1)^nz^{n+i\epsilon}q^{(n+i\epsilon)^2/2} \right).$$
\end{lemma}

\begin{theorem} \label{first}
Let ${\rm Re}(\epsilon_i)<0$ for all components of the vector $\epsilon$. Then 
\[
P_\epsilon\left(\frac{u}{\tau}, -\frac{1}{\tau}\right)=h_\epsilon(u, \tau).
\]
\end{theorem}
\begin{proof}
Using Gauss' integral, we find 
\begin{equation}\nonumber
\begin{split}
P_\epsilon\left(\frac{u}{\tau}, -\frac{1}{\tau}\right)&= e^{\frac{\pi i}{\tau}\left(u, Au\right)}  \sqrt{\frac{(-i\tau)^n}{ \det A}}  \sum_{k \in  \left( \Z_{\geq0}+\frac{1}{2}\right)^n} 
 \int_{\R^n} e^{2\pi \left(k, \epsilon\right)}  q^{\frac{1}{2}\left(w, A^{-1}w\right)}e^{2\pi i\left(u-k, w\right)}d^nw.
\end{split}
\end{equation}
If we can exchange the order of integration and summation, then the Lemma follows immediately (applying the substitution $w \mapsto -w$).
By Fubini's Theorem, we can indeed do this, since 
\begin{equation}\nonumber
\begin{split}
 \Big| e^{2\pi \left(k, \epsilon\right)} q^{\frac{1}{2}\left(w, A^{-1}w\right)}e^{2\pi i\left(u-k, w\right)} \Big|=
 \Big| e^{2\pi \left(k, \epsilon\right)}\Big|  |q|^{\frac{1}{2}\left(w, A^{-1}w\right)}\Big|e^{2\pi i\left(u-k, w\right)} \Big|
\end{split}
\end{equation}
and hence 
\begin{equation}\nonumber
\begin{split}
 \sum_{k \in  \left( \Z_{\geq0}+\frac{1}{2}\right)^n} 
  \int_{\R^n} \Big| e^{2\pi \left(k, \epsilon\right)} q^{\frac{1}{2}\left(w, A^{-1}w\right)}e^{2\pi i\left(u-k, w\right)} \Big|d^nw&= C_1C_2 < \infty, 
\end{split}
\end{equation}
where the numbers 
\begin{equation}\nonumber
\begin{split}
C_1 &= \int_{\R^n}  |q|^{\frac{1}{2}\left(w, A^{-1}w\right)}\Big|e^{-2\pi i\left(u, w\right)} \Big|d^nw <\infty\qquad\text{and} \qquad
C_2 =  \sum_{k \in  \left( \Z_{\geq0}+\frac{1}{2}\right)^n}  \Big| e^{2\pi \left(k, \epsilon\right)}\Big| <\infty
\end{split}
\end{equation}
are both finite, as $C_1$ is a convergent Gauss integral and $C_2$ converges because all components of $\epsilon$ have negative real part. 
\end{proof}
Observe that the previous result, if specialized to $n=1$, $A=1$, reproves a result from \cite{CM1}.
Substituting $P(u, \tau) = P_\epsilon\left(u+iA^{-1}\epsilon, \tau\right)$
in the Theorem above we get
\begin{corollary}
Let $\epsilon \in \mathbb C^n$ with $\text{Re}\left(\epsilon_i\right)<0$ for all components of $\epsilon$, then
\[
P\left(\frac{u}{\tau}, -\frac{1}{\tau}\right) =  e^{\frac{\pi i}{\tau}\left(u, Au\right)} (2i)^{-n} \sqrt{\frac{(-i\tau)^n}{ \det A}} \int_{\R^n-i\epsilon} \frac{q^{\frac{1}{2}\left(w, A^{-1}w\right)} e^{-2\pi i (u,w)} }{\prod\limits_{j=1}^n \left( \sin\left(\pi w_j \right)\right)} d^nw.
\]
\end{corollary}
We now have to evaluate a contour integral in order to get the $S$-transformation more explictely. 
For this, denote by $e_i$ in $\C^n$  the vector whose $i$-th component is one and all others are zero. Define $\epsilon_i$ via $\epsilon=\epsilon_1e_1+\dots+\epsilon_n e_n$. In order to employ previous results we need to choose $\alpha$ such that  ${\rm Re}(\epsilon-\alpha) \in \mathbb{R}_{<0}^n$.
For that define 
\[
\beta_i:= \left(1+\sgn\left(\re\left(\epsilon_i\right)\right)\right) \re\left(\epsilon_i\right)e_i, \qquad \alpha_i:= \sum_{j=1}^i \beta_j.
\]
with a warning that $\alpha_i$ does {\em not} denote the $i$-th coordinate of $\alpha$.
With this choice observe relation ${\rm Re}(\epsilon_i-\beta_i)=-{\rm sgn}({\rm Re}(\epsilon_i)){\rm Re}(\epsilon_i)<0$.

Let $\alpha_0:=0$ and define the contours $D_{\alpha_i}$ as the contours connecting the areas $\R^n+i\alpha_{i-1}$ (clockwise) and
$\R^n+i\alpha_{i}$ (counterclockwise), so that  $C_{\alpha_i}$, the contour as in Lemma \ref{lemma:translating}, satisfies $C_{\alpha_i}=D_{\alpha_1}\cup \dots\cup D_{\alpha_i}$.
Let 
\[
I:=\oint_{C_{\alpha_n}}   \frac{q^{\frac{1}{2}\left(w, A^{-1}w\right)} e^{-2\pi i (u,w)} }{\prod\limits_{j=1}^n \left( \sin\left(\pi\left(w_j-i\epsilon_j\right)\right)\right)} d^nw
\]
we also need to define (slightly modified) rank $r$ theta functions ($s=e_1+\dots+e_r$)
\[
\theta_{B, \epsilon}(u;\tau) := \sum_{n\in \Z^r+i\epsilon} (-1)^{(n-i\epsilon, s)} q^{\frac{1}{2}\left(n, Bn\right)} e^{2\pi i \left(u, n\right)} 
\]
here $B$ is a symmetric matrix with rational entries of size $r$, and $u, \epsilon$ in $\C^r$.
Further, let $w^{(i)}=w-(w, e_i)e_i$ the projection of $w$ on the hypersurface orthogonal to $e_i$. 
Then by Lemma \ref{one-dim-int} (for the third equation) applied for contour connecting ${\mathbb{R}}$ and $\mathbb{R}+i \beta_i$ at infinity, we get
\begin{equation}\nonumber
\begin{split}
I&= i\sum_{i=1}^n\left(1+\sgn\left(\re\left(\epsilon_i\right)\right)\right) I^{(1)}_i
\end{split}
\end{equation}
where
\begin{equation}\nonumber
\begin{split}
 I^{(1)}_i&= \oint_{D_{\alpha_i}}   \frac{q^{\frac{1}{2}\left(w, A^{-1}w\right)} e^{-2\pi i (u,w)} }{\prod\limits_{j=1}^n \left( \sin\left(\pi\left(w_j-i\epsilon_j\right)\right)\right)} d^nw\\
&=\int_{\R^{n-1}+i{\alpha_{i-1}}}   \theta _{\left(e_i,A^{-1}e_i\right), \epsilon_i}((e_i, A^{-1}w^{(i)})\tau-u_i  ;\tau) \frac{q^{\frac{1}{2}\left(w^{(i)}, A^{-1}w^{(i)}\right)} e^{-2\pi i (u,w^{(i)})} }{\prod\limits_{\substack{j=1\\ j\neq i}}^n \left( \sin\left(\pi\left(w^{(i)}_j-i\epsilon_j\right)\right)\right)} d^{n-1}w^{(i)}\\
&=\int_{\R^{n-1}}   \theta _{\left(e_i,A^{-1}e_i\right), \epsilon_i}((e_i, A^{-1}w^{(i)})\tau-u_i  ;\tau) \frac{q^{\frac{1}{2}\left(w^{(i)}, A^{-1}w^{(i)}\right)} e^{-2\pi i (u,w^{(i)})} }{\prod\limits_{\substack{j=1\\ j\neq i}}^n \left( \sin\left(\pi\left(w^{(i)}_j-i\epsilon_j\right)\right)\right)} d^{n-1}w^{(i)}\\
&\qquad + \oint_{C^{(1)}_{\alpha_{i-1}}}   \theta _{\left(e_i,A^{-1}e_i\right), \epsilon_i}((e_i, A^{-1}w^{(i)})\tau-u_i  ;\tau) \frac{q^{\frac{1}{2}\left(w^{(i)}, A^{-1}w^{(i)}\right)} e^{-2\pi i (u,w^{(i)})} }{\prod\limits_{\substack{j=1\\ j\neq i}}^n \left( \sin\left(\pi\left(w^{(i)}_j-i\epsilon_j\right)\right)\right)} d^{n-1}w^{(i)}
\end{split}
\end{equation}
where $C_{\alpha_i}^{(i)}$ is a contour connecting the areas of integration of lines three and four in above equation. 

It is now clear how to apply induction to get the desired result. For this, introduce for any nonempty subset $v$ of $S=\{1,...,n\}$ and $w$ in $\R^n$
the following quantities:
\begin{equation}\nonumber
\begin{split}
\qquad w^{(v)} &:= w -\sum_{i\in v} (w, e_i)e_i,\\
B^{(v)}&:= \left(B_{ab}\right)_{a, b \in v},\qquad
B_{ab}:= \left(v_a e_a, A^{-1} v_b e_b\right), \\
u^{(v)}&:= \sum_{i\in v} (u, e_i)e_i, \qquad \epsilon^{(v)}:= \sum_{i\in v} (\epsilon, e_i)e_i, \qquad x^{(v)}\left(w^{(v)}\right):= \sum_{i\in v} \left( v_i e_i, A^{-1}w^{(v)}\right) e_i. \\ 
\end{split}
\end{equation}
By induction on $|v|$, we get
\begin{equation}\nonumber
\begin{split}
I &= \sum_{v \in \mathcal{P}(S) \setminus \emptyset  }  (i)^{|v|}\left(\prod_{j\in v}   \left(1+\sgn\left(\re\left(\epsilon_j \right)\right)\right)\right) I^{(v)} \\
I^{(v)} &= \int_{\R^{n-|v|}}   \theta _{B^{(v)}, \epsilon^{(v)}}\left(x^{(v)}\left(w^{(v)}\right)\tau-u^{(v)}  ;\tau\right) \frac{q^{\frac{1}{2}\left(w^{(v)}, A^{-1}w^{(v)}\right)} e^{-2\pi i (u,w^{(v)})} }{\prod\limits_{\substack{j=1\\ j\not\in v}}^n \left( \sin\left(\pi\left(w^{(v)}_j-i\epsilon_j\right)\right)\right)} d^{n-|v|}w^{(v)}.\\
\end{split}
\end{equation}
This together with Lemma \ref{lemma:translating} for $\alpha=\alpha_n$ then proves that

\begin{theorem} \label{S-partial} Let $I$ be as above and $\epsilon_j \notin i \mathbb{R}$, $j=1,...,n$ . Then 
\begin{equation}\label{eq:finalS}
P_\epsilon\left(\frac{u}{\tau}, -\frac{1}{\tau}\right)=h_\epsilon(u, \tau)+e^{\frac{\pi i}{\tau}\left(u, Au\right)} (2i)^{-n} \sqrt{\frac{(-i\tau)^n}{ \det A}} I.
\end{equation}
\end{theorem}

\section{Kostant's partial theta functions}

Let $Q$ be a root lattice of type $ADE$ with $\mathbb{Z}$-basis $\{ \alpha_i\}_{i=1}^n$ with the Gramm matrix $A$. As in Section 1,  $( \cdot, \cdot)$ denote the standard 
bilinear form on $\mathbb{C}^n$. 
%Sometimes we shall identify the basis elements of ${Q}$ with the standard basis in $\mathbb{C}^n$. This way the bilinear form is also defined on $Q$. We denote by $\frak{g}$ the simple Lie algebra associated to $Q$. 
Let $\Delta_+$ denote the set of positive roots.
Consider
\[
K_\epsilon(u, \tau) := \sum_{k \in  \left( \Z_{\geq0}+\frac{1}{2}\right)^n} K \left((k_1-\frac{1}{2}) \alpha_1+\cdots + (k_n-\frac{1}{2}) \alpha_n \right) q^{\frac{1}{2} \left(k, Ak\right)} e^{2\pi i \left(k, Au\right)}e^{2\pi \left(k, \epsilon\right)},
\]
where 
$$\frac{1}{\prod_{\alpha \in \Delta_+} (1-e^{\alpha})}=\sum_{(k_1,...,k_n) \in \mathbb{Z}^n_{\geq 0}} K(k_1 \alpha_1 +\cdots + k_n \alpha_n) e^{k_1 \alpha_1+\cdots +k_n \alpha_n}$$
is the generating series for the Kostant partition function.
Also, define the integral
\[
k_\epsilon(u, \tau) :=  \frac{e^{\frac{\pi i}{\tau}\left(u, Au\right)}}{(2i)^{|\Delta_+|}} \sqrt{\frac{(-i\tau)^n}{ \det A}} \int_{\R^n} \frac{ q^{\frac{1}{2}\left(w, A^{-1}w\right)} e^{2\pi i (u,w)} }{ \Delta(w+i\epsilon)} d^nw,
\]
where  
\[
\frac{1}{\Delta(w+i\epsilon)}= \frac{e^{\pi i(\sum_i w_i+i \sum_i \epsilon_i)}  }{ e^{2 \pi i \rho_{w+i \epsilon}} \prod\limits_{\alpha \in \Delta_+}   \sin\left(\pi\left(\alpha, w+i\epsilon \right)\right)}
\]
and
\[\rho_{v}=\frac{1}{2} \sum_{\alpha \in \Delta_+} (\alpha,   v  ).
\] 
%and 
%{\bf AM: We are using two different $\Delta$ so perhaps we should just denote the denominator with $\Delta(w- i \epsilon)$, after we normalize. TC: I don't understand your comment here. }
%\[
%\frac{1}{\Delta(w+i\epsilon)}= \frac{e^{\pi i(\sum_i w_i+i \sum_i \epsilon_i)}  }{ e^{2 \pi i \rho_{w+i \epsilon}} \prod\limits_{\alpha \in \Delta_+}   \sin\left(\pi\left(\alpha, w+i\epsilon \right)\right)}.
%t\]
%{\bf AM:  I think we are using $( \cdot, \cdot)$ for two different things. Or we have to specify 
%$\alpha \in \mathbb{C}^n$. TC:  The $A^{-1}$ had to be removed in the inner products. }

\begin{lemma} \label{negative}
Let $Re(\epsilon_i)<0$, for all components of the vector $\epsilon$. Then 
\[
K_\epsilon\left(\frac{u}{\tau}, -\frac{1}{\tau}\right)=k_\epsilon(u, \tau).
\]
\end{lemma}
\begin{proof}
By using Gauss' integral, we find
\begin{equation}\nonumber
\begin{split}
K_\epsilon\left(\frac{u}{\tau}, -\frac{1}{\tau}\right)&= e^{\frac{\pi i}{\tau}\left(u, Au\right)}  \sqrt{\frac{(-i\tau)^n}{ \det A}}  \sum_{k \in  \left( \Z_{\geq0}+\frac{1}{2}\right)^n} 
K \left((k_1-\frac{1}{2}) \alpha_1+\cdots + (k_n-\frac{1}{2}) \alpha_n \right) \cdot \\
&\qquad \int_{\R^n} e^{2\pi \left(k, \epsilon\right)}  q^{\frac{1}{2}\left(w, A^{-1}w\right)}e^{2\pi i\left(u-k, w\right)}d^nw \\
&= \frac{e^{\frac{\pi i}{\tau}\left(u, Au\right)}}{(2i)^{|\Delta_+|}} \sqrt{\frac{(-i\tau)^n}{ \det A}} \int_{\R^n} \frac{q^{\frac{1}{2}\left(w, A^{-1}w\right)} e^{2\pi i (u,w)} }{\prod\limits_{r_1 \alpha_1+\cdots +r_n \alpha_n \in \Delta_+}  \left( \sin\left(\pi\left(\sum_{i} r_i w_i+i \sum_{i} r_i \epsilon_i \right)\right)\right)}\\
&=\frac{e^{\frac{\pi i}{\tau}\left(u, Au\right)}}{(2i)^{|\Delta_+|}} \sqrt{\frac{(-i\tau)^n}{ \det A}} \int_{\R^n} \frac{   e^{\pi i(\sum_i w_i+i \sum_i \epsilon_i)} q^{\frac{1}{2}\left(w, A^{-1}w\right)} e^{2\pi i (u,w)} }{\prod\limits_{\alpha \in \Delta_+} e^{\pi i (\alpha, w+i \epsilon)} \left( \sin\left(\pi\left(\alpha,  w+i\epsilon \right)\right)\right)} d^nw.
\end{split}
\end{equation}
Changing the order of integration and summation can be done by the same argument as in Lemma 5. 
\end{proof}
The proof of the following result is the same as the one for Lemma \ref{lemma:translating}
\begin{lemma}
Let $\alpha \in \mathbb C^n$ such that $\epsilon-\alpha\in \left(\C \setminus i\R\right)^n$ 
and let $u'=u-iA^{-1}\alpha\tau$. Then
\[
K_\epsilon\left(\frac{u}{\tau}, -\frac{1}{\tau}\right)= K_{\epsilon-\alpha}\left(\frac{u'}{\tau}, -\frac{1}{\tau}\right),\
\]
and 
\[
k_{\epsilon-\alpha}(u', \tau) = k_{\epsilon}(u, \tau) +
\frac{e^{\frac{\pi i}{\tau}\left(u, Au\right)}}{(2i)^{|\Delta_+|}} \sqrt{\frac{(-i\tau)^n}{ \det A}}\oint_{C_\alpha}   \frac{  q^{\frac{1}{2}\left(w, A^{-1}w\right)} e^{2\pi i (u,w)} }{ \Delta(w+i\epsilon)} d^nw,
\]
where the contour $C_\alpha$ connects the regions $\R^n$ (clockwise) and $\R^n+i\alpha$ (counter clockwise) at infinity.  
For $\alpha \in i \mathbb{R}^n $ the contour integral vanishes. 
\end{lemma}
For a given example it is now a tedious exercise to evaluate these contours. 
This will be illustrated for the $\text{Re}(\epsilon_i)>0$ region and $A_n$ root lattices.

\subsection{Modular transformations: the $\text{Re}(\epsilon_i)>0$ region.}

Let us consider the example where $\text{Re}(\epsilon_i)>0$ for all $i$ and $Q=A_n$ ($A$-type). 
We will now refine the contours we have used for partial theta functions. But otherwise the startegy follows closely the partial theta function case. 
Let
\[
\nu = \frac{1}{2} \text{min}\left\{|{\rm Re}(\epsilon_i)|, i=1, \dots, n\right\}
\]
and define 
\[
\beta_i:= ({\rm Re}(\epsilon_i)+\nu) e_i, \qquad \gamma_i:= \sum_{j=1}^i \beta_j,
\]
with the same warning as before that $\gamma_i$ does {\em not} denote the $i$-th coordinate of $\gamma$.
With this choice observe the relation ${\rm Re}(\epsilon_i-\beta_i)<0$.
Let $\gamma_0:=0$ and define the contours $D_{\gamma_i}$ as the contours connecting the areas $\R^n+i\gamma_{i-1}$ (clockwise) and
$\R^n+i\gamma_{i}$ (counterclockwise), so that  $C_{\gamma_i}$, the contour as in the Lemma, satisfies $C_{\gamma_i}=D_{\gamma_1}\cup \dots\cup D_{\gamma_i}$.
If we omit the integration along $\mathbb R$ for the $j$-th coordinate for $j>i$, we denote the corresponding contour $C_{j, \gamma_i}$
and similarly $D_{j, \gamma_i}$. If we omit more than one direction, we indicate it by listing the corresponding indices. 

Let $f=f(w_1, \dots, w_n, u_1, \dots, u_n)$ and define the theta-like function associated to $f$ in the $i$-th direction as 
\[
\theta^{\{i\}}(f):= \sum_{n\in \mathbb Z} (-1)^n f\big\vert_{w_i=n-i\epsilon_i} 
\]
and more general for a subset $S$ of $\{1, \dots, n\}$ we define the theta-like function in the directions $S$ as 
\[
\theta^{(S)}\left( f \right) := \sum_{\substack{n_s \in\mathbb Z \\ s\in S}}(-1)^{n_s} f\big\vert_{\left\{w_s=n_s-i\epsilon_s|s\in S\right\}}. 
\]
Let us also introduce the short-hand notation
\[
d^{n-\ell}_{s_1, \dots, s_\ell} := dw_1\dots \widehat{dw_{s_1}}\dots \widehat{dw_{s_\ell}}\dots dw_n
\]
where we omit differentials $dw_{s_1}, \dots, dw_{s_\ell}$. The subspace of all vectors $w$ with the property that $w_{s_1}=\dots=w_{s_\ell}=0$ then embeds $\mathbb R^{n-\ell}$ in $\mathbb R^n$.
We chose to use this notation, since by Cauchey's residue theorem 
\begin{equation}\nonumber
\begin{split}
\oint_{D_{\gamma_i}} \frac{f(w_1, \dots, w_n, u_1, \dots, u_n)}{\sin(\pi(w_i+i \epsilon_i))}dw^n &= 
2i \int_{\mathbb R^{n-1}+i\gamma_{i-1}} \theta^{\{i\}}(f)dw^{n-1}_i\\
&=2i \oint_{C_{i, \gamma_{i-1}}} \theta^{\{i\}}(f)dw^{n-1}_i+
 2i \int_{\mathbb R^{n-1}} \theta^{\{i\}}(f)dw^{n-1}_i
\end{split}
\end{equation}
We can now compute (we write $f$ for $f(w_1, \dots, w_n, u_1, \dots, u_n)$)
\begin{equation}\nonumber
\begin{split}
\oint_{D_{\gamma_i}} &\frac{f}{\sin(\pi(w_i-i \epsilon_i))^m}dw^n = 
\frac{e^{m\pi \epsilon_i}}{(m-1)!}    \left(\frac{e^{2\pi \epsilon_i}}{-\pi i}\frac{d}{d\epsilon_i}\right)^{m-1}
 \oint_{D_{\gamma_i}} \frac{f e^{-(m-1)\pi i w_i}}{\sin(\pi(w_i+i \epsilon_i))} dw^n\\
&= 2i\frac{e^{m\pi \epsilon_i}}{(m-1)!}    \left(\frac{e^{2\pi \epsilon_i}}{-\pi i}\frac{d}{d\epsilon_i}\right)^{m-1}\Bigg(  \oint_{C_{i, \gamma_{i-1}}} \theta^{\{i\}}(fe^{-(m-1)\pi i w_i})dw^{n-1}_i+\\
&\qquad  \int_{\mathbb R^{n-1}} \theta^{\{i\}}(fe^{-(m-1)\pi i w_i})dw^{n-1}_i\Bigg)
\end{split}
\end{equation}
Introducing another short-hand notation
\[
D_{\epsilon_i}^m := 2i\frac{e^{(m+1)\pi \epsilon_i}}{m!}    \left(\frac{e^{2\pi \epsilon_i}}{-\pi i}\frac{d}{d\epsilon_i}\right)^{m}
\]
this can be more compactly written as
\begin{equation}
\begin{split}
\oint_{D_{\gamma_i}} \frac{f}{\sin(\pi(w_i+i \epsilon_i))^m}dw^n &= D_{\epsilon_i}^{m-1}\Bigg( \oint_{C_{i, \gamma_{i-1}}} \theta^{\{i\}}(fe^{-(m-1)\pi i w_i})dw^{n-1}_i+\\
&\qquad  \int_{\mathbb R^{n-1}} \theta^{\{i\}}(fe^{-(m-1)\pi i w_i})dw^{n-1}_i\Bigg)
\end{split}
\end{equation}
We can now evaluate the contour integrals: For each set $R\subset P_n$, where $P_n$ is the power set of $\{1, \dots, n\}$, a contribution. Namely $R=\left\{r_1>r_2>\dots >r_m\right\}$ corresponds to the contours $D_{\gamma_{r_1}}, D_{r_1, \gamma_{r_2}}, D_{r_1, r_2, \gamma_{r_3}}, \dots$.

The denominator satisfies
\[
\frac{\sin\left(\pi \left(  w_s+i\epsilon_s   \right)\right)}{\Delta\left(w+i\epsilon\right)}\Bigg\vert_{w_s+i\epsilon_s=m} =  \frac{\sin\left(\pi \left(  w_s+i\epsilon_s   \right)\right)}{\Delta\left(w+i\epsilon\right)}\Bigg\vert_{w_s+i\epsilon_s=0} 
\]
for every integer $m$. 
%{\bf AM: But $sin(m \pi +r)=-sin(r)$ for $m$ odd. TC: The statement is correct. The minus signs from the sign cancel with the one from $e^{2\pi i \rho_w}$ and then the exponential in the %nominator makes the sign correct.}
Let $m_{r_1}=1$
and
\[
\frac{1}{\Delta^{(r_1, m_{r_1})}(w+i\epsilon)} := \frac{ \sin\left(\pi \left(  w_{r_1}+i\epsilon_{r_1}   \right)\right)}{\Delta\left(w_1+i\epsilon_1, \dots, w_n+i\epsilon_n\right)}\Bigg\vert_{w_{r_1}+i\epsilon_{r_1}=0}.
\]
We now define quantities $m_{r_{i+1}}$ and $\left(\Delta^{(r_1, m_{r_1}),\dots, (r_{i+1}, m_{r_{i+1})}}(w+i\epsilon)\right)^{-1}$ recursively.
First we set $m_{r_{i+1}}+1$ as the pole order of 
\[
\frac{1}{\Delta^{(r_1, m_{r_1}),\dots, (r_i, m_{r_i})}(w+i\epsilon)}
\]
at $w_{r_{i+1}}+i\epsilon_{r_{i+1}}$ at zero and then
\[
\frac{1}{\Delta^{(r_1, m_{r_1}),\dots, (r_{i+1}, m_{r_{i+1})}}(w+i\epsilon)}:=\frac{\sin\left(\pi \left(  w_{r_{i+1}}+i\epsilon_{r_{i+1}}   \right)\right)^{m_{r_{i+1}+1}}}{\Delta^{(r_1, m_{r_1}),\dots, (r_i, m_{r_i})}(w+i\epsilon)}\Bigg\vert_{w_{r_{i+1}}+i\epsilon_{r_{i+1}}=0}.
\]
Note, that in the type $A$ case $m_{r_i}= \text{min} \left\{ m \in\mathbb Z_{\geq 0} \ | \ r_i+m+1 \notin R\right\}$.
The contribution of $R$ to the transformation of $K_\epsilon$ is then
% {\bf AM:  Do we want to put $\epsilon_1,..,\epsilon_r$ inside the expression $\theta^{r_1,...,r_\ell}( \cdot)$? We have to make it consistent and I think it is better readable if we leave them out; so %I did this.}
\begin{equation}\nonumber
\begin{split}
X_R &=  \sum_{\ell=1}^{|R|} \prod_{j=1}^\ell D_{\epsilon_{r_j}}^{m_{r_j}} \int_{\mathbb R^{n-\ell}} 
\frac{\theta^{\{r_1, \dots, r_\ell\}}\left(q^{\frac{1}{2}\left(w, A^{-1}w\right)}e^{2\pi i (u, w)} \prod\limits_{k=1}^\ell e^{\pi i\left(w_{r_k}+m_{r_k}i\epsilon_{r_k}\right)} \right)}{\Delta^{(r_1, m_{r_1}), \dots, (r_\ell, m_{r_\ell})}(w+i\epsilon)} dw^{n-\ell}_{r_1,\dots, r_\ell}
\end{split}
\end{equation}
It follows that
\begin{equation}\label{K-formula}
K_\epsilon\left(\frac{u}{\tau}, -\frac{1}{\tau}\right) = \frac{e^{\frac{\pi i}{\tau}\left(u, Au\right)}}{(2i)^{|\Delta_+|}} \sqrt{\frac{(-i\tau)^n}{ \det A}}
\sum_{R\in P_n} X_R.
\end{equation}
Note, that
\[
X_{\emptyset}=k_\epsilon(u, \tau)
\]
and that $X_{\{1, \dots, n\}}$ contains the summand
\begin{equation}\label{eq:Y}
\begin{split}
e^{2\pi (u, \epsilon)}Y_\epsilon(u, \tau) :&= \prod_{j=1}^n D_{\epsilon_{j}}^{n-j} e^{-\pi i \sum\limits_{j=1}^n (1-n+j)\epsilon_j} \sum_{k\in\mathbb Z^n} q^{-\frac{1}{2} \left((\epsilon + ik), A^{-1}(\epsilon+ik)\right))} e^{-2\pi i (u, k)}\\
&= \prod_{j=1}^n D_{\epsilon_{j}}^{n-j} e^{-\pi i \sum\limits_{j=1}^n (1-n+j)\epsilon_j} q^{-\frac{1}{2}\left(\epsilon, A^{-1}\epsilon\right)}  \theta_{A_n^{-1}}(\tau,   u+i\tau A^{-1}\epsilon)
\end{split}
\end{equation}
that is a higher derivative of the theta function of the weight lattice (suitably $\epsilon$-regularized).

\section{Characters of $W(p)_Q$ and of $W^0(p)_Q$ modules}

In this part we discuss two vertex algebras and their modules following mostly \cite{M1} and \cite{BM2}.
% {\bf TC: please add [M1], [FrB] to the references}
As before we denote the root lattice of type $ADE$ by $Q$, $P=Q^0$ its weight lattice, $L=\sqrt{p} Q$ a dilated root lattice,  
and $L^0$ its dual. Also, $P_+$ denotes the intersection of $P$ with the fundamental Weyl chamber. 
We fix simple roots $\alpha_1$,...,$\alpha_n$ and denote by $\omega_1,...,\omega_n$ the corresponding fundamental weights, which also generate the monoid $P_+$.
For $\lambda \in \frak{h}^*$, where $\frak{h}=\mathbb{C} \otimes_{\mathbb{Z}} Q$, we shall denote by $F_\lambda$ the rank $n$ Fock space
with highest weight $\lambda$. 

We assume here some knowledge of affine $W$-algebras as presented in \cite{FrB} and references therein. In particular, construction via screening operators
(see also \cite{AM,FT}). We stress that these results are not needed for understanding the main results of the paper, which are purely of analytic nature.
The next result seems to be known \cite{FT} (see also \cite{AM} where it was also mentioned). However, we could not find a proof in the literature.
\begin{theorem} \label{non-generic} Let $\frak{g}$ be simply laced. Then $p=k+h^\vee \in \mathbb{N}_{\geq 2}$ is non-generic. More precisely,
$${W}^0(p)_Q:=\bigcap_{j=1}^n {\rm Ker}_{{F}_0} e^{-\alpha_j/\sqrt{p}}_0$$
is a vertex algebra containing the universal affine $W$-algebra  $\mathcal{W}_p(\frak{g})$-algebra as a proper subalgebra. 
%In particular, for $Q=A_1$ this vertex algebra is the $(1,p)$-singlet algebra  discussed in \cite{CM} and \cite{CMW}.
\end{theorem}
\begin{proof} We first prove that the corresponding universal $W$-algebra $W_k(\frak{g})$ is simple.
Feigin-Fuchs' duality \cite{Ar2, BFN} of $W$-algebras gives $W_k(\frak{g}) \cong W_k'(\frak{g}^L)$, whenever $1=(k+h^\vee)(k'+h^\vee)$. 
For ADE type the Langlands dual $\frak{g}^L$ is $\frak{g}$. So instead we can look at $W_{k'}(\frak{g})$. Result in \cite{GK} proves irreducibility of the vacuum module $V_\frak{g}(k' \Lambda_0)$. Using the Drinfeld-Sokolov reduction result of Feigin and Frenkel $W_{k'}(g)$ is also simple. 
As it is known universal affine $W$-algebra acts on the Fock space $F_0$. Because of the simplicity, the cyclic $W_k(p)$-modules generated by $ {\bf 1} \in F_0$ 
is isomorphic to $W_k(p)$. To see that the kernel is bigger than $W_k(\frak{g})$, by using Lemma 6.1 \cite{MP} or \cite{FT}, we see that  
$$e^{\sqrt{p} \alpha_i}_0 e^{-\sqrt{p} \alpha_i}, \ \ i=1,...,n,$$
belongs to $W^0(p)_Q$ but not in $W_k(\frak{g})$ \cite{FT}. 
\end{proof}

The previous vertex algebra can be maximally extended leading to
\begin{equation} \label{short}
{W}(p)_Q:=\bigcap_{j=1}^n {\rm Ker}_{V_L} e^{-\alpha_j/\sqrt{p}}_0.
\end{equation}
Again, if we let $Q=A_1$ we recover the well-known triplet vertex algebra \cite{AM4, FGST} usually denoted by ${W}(p)$.
%Similarly, we can construct ${W}(p)_Q$-modules inside the appropriate Fock spaces. 
Observe now that ${W}^0(p)_Q$ can be viewed as the $U(1)^n$-invariant subalgebra of 
${W}(p)_Q$, where the action of $U(1)$ is obtained from the rank-one Heisenberg VOA.

First, we discuss certain $W(p)_Q$-modules following \cite{FT}. These are constructed as follows.
We choose $\lambda_j=\frac{\omega_j}{\sqrt{p}}$, $j=1,...,n$ to be a basis of $L^0$.
We also fix our representative of the congruence classes of $Q^0/Q$ to be $0,\omega_1,...,\omega_n$ for type $A_n$,  $0$, $\omega_{n-1}$ and $\omega_n$ in type $D_n$, and similarly for
type $E$. Now, following \cite{FT} each coset  $L^0/L$ has a unique representative $\lambda$ 
$$\lambda=\sqrt{p} \hat{\lambda}+\sum_{j=1}^n(1-s_j)\lambda_j,$$
where $\hat{\lambda}$ is a representative of $Q^0/Q$ fixed earlier, and $\bar{\lambda}=\sum_{j=1}^n(1-s_j)\lambda_j$ such that  
$s_j \in \{1,2,...,p \}$.  Observe that even for $\lambda \in L^0$ this representation is unique if $\hat{\lambda} \in P=Q^0$. 

\subsection{Characters of $W(p)_Q$-modules}

In this section $\delta(z)= \prod_{\alpha \in \Delta_-} (1-z^{\alpha})$ or equivalently $z^{2 \rho} \delta(z)=(-1)^{|\Delta_+|} \prod_{\alpha \in \Delta_+} (1-z^{\alpha})$, denotes 
the Weyl denominator.

For each $\lambda \in L^0/L$ we now associate, conjecturally, an irreducible $W(p)_Q$-character and its full $(\tau,z)$-character; $p^{{\rm rank}(Q)} |Q^0/Q|$ modules and characters in total. We omit discussion of modules as they still only conjecturally correspond to the characters below (see \cite{FT}). Modules will be denoted by $W(p,\lambda)_Q$, where $\lambda \in L^0/L$ with fixed representatives as above. Then the full-character is given by \cite{FT} (see also \cite{BM2} for a related discussion)
\[ {\rm ch}[W(p,\lambda)_Q](\tau,{ z})=\frac{1}{\eta(\tau)^{\text{rank}(Q)}} \sum_{w \in W} \sum_{\alpha \in Q} (-1)^{l(w)} \frac{q^{\frac{1}{2}|| \sqrt{p} \alpha +\lambda  +\left(\sqrt{p}-\frac{1}{\sqrt{p}}\right) \rho ||^2}{\bf z}^{w\left(\alpha+  \hat{\lambda}\right)}}{ w(\delta({ z}))}
\]
\begin{equation} \label{keys}
=\frac{1}{\eta(\tau)^{\text{rank}(Q)}}  \sum_{\alpha \in Q }  q^{\frac{1}{2} || \sqrt{p} \left(\alpha+\rho+\hat{\lambda}\right)+\bar{\lambda} -\frac{1}{\sqrt{p}} \rho ||^2} \left( \sum_{w \in W } (-1)^{l(w)} \frac{{\bf z}^{w\left(\alpha+\rho+\hat{\lambda}\right)-\rho}}{\delta({ z})} \right).
\end{equation}
We stress that the $z$-variable is not visible from the inner structure of $W^0(p)_Q$ and that this formula, after we specialize $z=0$, is only conjecturally character of an irreducible 
$W(p)_Q$-module. Then we have  \cite{FT,BM1} 
\begin{proposition}
\begin{equation} \label{keys2}
{\rm ch}[W(p,\lambda)_Q](\tau,z)=\sum_{\alpha \in Q \cap P^+} \chi_{\hat{\lambda}+\alpha}(z) \left( \sum_{w \in W}  (-1)^{l(w)} 
\frac{q^{\frac{1}{2}\left|\left| \sqrt{p} w\left(\alpha+\rho+\hat{\lambda}\right)+\bar{\lambda} -\frac{1}{\sqrt{p}} \rho \right|\right|^2}}
{\eta(\tau)^{\text{rank}(Q)}} \right),
\end{equation}
where $\chi_{\hat{\lambda}+\alpha}(z)$ denote the character of the finite-dimensional $\frak{g}$-module of highest weight $\hat{\lambda}+\alpha$, $\ell(w)$ is the length of $w \in W,$ an element in the (finite) Weyl group.
\end{proposition}
After we specialize at ${\bf z}=1$, we get
\begin{equation} \label{keys3}
{\rm ch}[W(p,\lambda)_Q](\tau)=\sum_{\alpha \in Q \cap P^+} {\rm dim} \left( V \left(\hat{\lambda}+\alpha\right) \right) \left( \sum_{w \in W}  (-1)^{l(w)} 
\frac{q^{\frac{1}{2}\left|\left| \sqrt{p} w\left(\alpha+\rho+\hat{\lambda}\right)+\bar{\lambda} -\frac{1}{\sqrt{p}} \rho \right|\right|^2}}
{\eta(\tau)^{\text{rank}(Q)}} \right),
\end{equation}
where $V(\hat{\lambda}+\alpha)$ denotes the irreducible $\frak{g}$-module of highest weight $\hat{\lambda}+\alpha$.
\begin{remark}{\em 
Modularity of ${\rm ch}[W(p,\lambda)_Q](\tau)$ was studied in \cite{BM2}. It was proven that 
the function $$\eta(\tau)^{rank(Q)} {\rm ch}[W(p,\lambda)_Q](\tau)$$ is a generalized modular form, in the sense that it can be written as a linear combination 
of modular forms of different non-negative integral weight with respect to a congruence subgroup. The highest weight component of 
this modular form is of weight $|\Delta^+|$. Consequently, ${\rm ch}[W(p,\lambda)_Q](\tau)$, $\lambda \in L^0/L$, combine into a larger {logarithmic} 
vector-valued modular form as conjectured in \cite{FT}.}
\end{remark}

% We shall assume that all atypical (that is, those that are not isomorphic to Verma modules) 
%arise in this way. Now it's easy to write down explicit formula for characters by simply extracting coefficient in the $z$-expansion of 

\subsection{Typical $W^0(p)_Q$-modules and their characters}

In parallel to \cite{BM1, CM1,BM2}, we are only interested in atypical and typical $W^0(p)_Q$-modules. 
Because the vertex operator algebra ${W}^0(p)_Q$ is a subalgebra of the vacuum Fock space $F_0$, it cannot be $C_2$-cofinite as it admits uncountably many irreducible modules
up to equivalence. These are constructed from $F_{\lambda}$, $\lambda \in \frak{h}$  and their sub-quotients. As a reminder, here $F_\lambda$ denotes the highest weight $F_0$-module 
generated by $v_\lambda$ such that $a \cdot v_{\lambda}=\langle a, \lambda \rangle v_{\lambda}$ for all $a \in \frak{h}$.
In particular, countably many ${W}^0(p)_Q$-many modules are obtained by restriction from  $W(p)_Q$-modules. Provided the conjectural character formulas in \cite{FT} are correct, this will indeed 
give correct characters of ${W}^0(p)_Q$-modules. These modules/characters will be called atypical . 
Typical representations are those isomorphic to Fock spaces which do not arise from $W(p)_Q$-modules. In other words, they are given by $F_{\lambda}$, where 
$\lambda \notin L^0$. Because our choice of conformal vector $\omega \in W^0(p)_Q$ is chosen according to \cite{AM,FT,BM2} we can easily compute:
For $\lambda \in \mathbb{C} \otimes_{\mathbb{Z}} L$, we have
\begin{equation} \label{char-generic} 
{\rm ch}[F_{\lambda}](\tau)=\frac{q^{||\lambda-( \sqrt{p}-\frac{1}{\sqrt{p}} ) \rho||^2/2}}{\eta(\tau)^n}.
\end{equation}

\begin{remark}
For ${\rm rank}(L)=1$ with $\rho=\frac{\alpha}{2}$ and $<\alpha,\alpha>=2$, if we take $\varphi=\frac{\alpha}{\sqrt{2}}$ 
as we did in \cite{CM1}, and $\varphi (0) v_{\lambda}=\lambda v_{\lambda}$, then the above  
formula (\ref{char-generic})  gives the  formula in \cite{CM1}: For $\lambda \in \mathbb{C}$ we get
\begin{equation}
{\rm ch}[F_{\lambda}](\tau)=\frac{q^{(\lambda-\alpha_0/2)^2/2}}{\eta(\tau)}.
\end{equation}

\end{remark}

\begin{proposition} Each typical modules $F_{\lambda}$, $\lambda \notin L^0$ is irreducible  as an affine $W$-algebra module, and thus as a $W^0(Q)_p$-module.
\end{proposition}
\begin{proof}
This is a consequence of Theorem 6.3.1 in \cite{Ar1}. Viewed as a module for the affine $W$-algebra, the highest weight of $F_{\lambda}$  satisfies conditions in the theorem.
Comparing characters of the Verma module and $F_{\lambda}$ now yields the claim. 
\end{proof}

\subsection{Atypical modules and characters}

Characters of atypical characters of $W^0(p)_Q$-modules are now computed by extracting an appropriate power of ${\bf z}$. 
There are some subtleties the way this is done so we explain the procedure. Let $\mu \in L^0$, then there is a unique 
representative $\lambda \in L^0/L$ and $\beta \in Q$ such that $\mu=\lambda+\sqrt{p} \gamma$. 
Because the $z$-coefficients of ${\rm ch}[W(p,\lambda)_Q]$ lie inside $Q+\hat{\lambda}$, after we let $\gamma'=\gamma+\hat{\lambda}$, we  
get a non-trivial constant term
\begin{equation} 
{\rm ch}[W^0(p,\mu)_Q](q):={\rm CT}_{\bf z} \left\{ {\bf z}^{-\gamma'} {\rm ch}[{W(p,\lambda)_Q}](\tau, {z})\right\}.
\end{equation}
The most important case occurs of course for  $\lambda=0$ and $\beta=0$; this gives the character of the vertex algebra $W^0(p)_Q$: 
\begin{align}
{\rm ch}[W^0(p)_Q)](q)&={\rm CT}_{\bf z}  \{ {\rm ch}[{W(p)_Q}](\tau,{ z}) \} \nonumber \\
&=\sum_{\alpha \in Q \cap P^+}  \mathrm{dim}(V(\alpha)_0)  \left( \sum_{w \in W}  (-1)^{l(w)} \frac{q^{\tiny{\frac{1}{2} || \sqrt{p} w(\alpha+\rho) -\frac{1}{\sqrt{p}} \rho ||^2}}}{\eta(\tau)^{\text{rank}(Q)}} \right),
\end{align}
where $V(\alpha)_0$ denote the zero weight subspace of $V(\alpha)$. Again, we stress that this formula is correct provided that conjectural character formulas of $W(p)_Q$-modules
are correct.
In view of relation (\ref{keys2}), it is clear that characters  ${\rm ch}[W^0(p,\mu)](\tau)$ are also invariant under the Weyl group, meaning that 
\begin{equation} \label{weyl-invariance}
{\rm ch}[W^0(p,\mu)](\tau)={ \rm ch}[W^0(p,w.\mu)](\tau); \ \ w \in W
\end{equation}
where $w.\mu=w(\mu+\rho)-\rho$ is the shifted action. Observe that this was also observed in a special case of $\frak{g}=sl_2$ in \cite{CM1}, where two irreducible 
modules have the same character.

Let us illustrate how this setup recovers characters  studied in our earlier work \cite{BM1}. Following the previous notation adopted for $A_1$, 
 for $-(p-1) \leq j \leq 0$, $\widehat{\lambda}=0$, $k \in \mathbb{Z}$, we get 
\begin{equation} \label{character-singlet1}
{\rm ch}\left[{W^0\left(p,\frac{j \omega}{\sqrt{p}} + k \sqrt{p} \alpha \right)_{A_1}}\right](\tau)=\frac{1}{\eta (\tau)} \sum_{n \geq k}^\infty  \left({q^{p\left(n+\frac{p+j-1}{2p}\right)^2}-q^{p\left(n+\frac{p-j+1}{2p}\right)^2}}\right),
\end{equation}
while for $ 1 \leq j \leq p$, $\widehat{\lambda}=\omega$, $k \in \mathbb{Z}$, we obtain 
\begin{equation} \label{character-singlet2}
{\rm ch}\left[W^0\left(p,\sqrt{p} \omega + \frac{j \omega}{\sqrt{p}}+k \sqrt{p} \alpha \right)_{A_1}\right](\tau)=\frac{1}{\eta (\tau)}\sum_{n \geq k}^\infty  \left({q^{p\left(n+\frac{2p-j-1}{2p}\right)^2}-q^{p\left(n+\frac{2p+j+1}{2p}\right)^2}}\right).
\end{equation}
Here as usual $\omega=\frac{\alpha}{2}$ denotes the fundamental weight of $\frak{sl}_2$.

\subsection{Atypical characters and Kostant's partial theta functions}

%Let $\lambda \in L^0/L$ be a coset representative as defined earlier. Of course any element $\mu \in L^0$ ca be expressed as $\lambda + \sqrt{p} \beta$
%where $\beta \in Q$ and $\lambda$ is the representative which we fixed (and thus $\hat{\lambda}$). 
%Observe that the $z$-powers of ${\rm ch}[W(p,\lambda)_Q](\tau,z)$ are contained in the set $\hat{\lambda}+Q \subset P$.
Again, for  $\mu \in L^0$. As explained above the character ${\rm ch}[W^0(p,\mu)](\tau)$ is obtained by taking 
the constant term of 
 \[ z^{-\gamma'} \frac{1}{\eta(\tau)^{\text{rank}(Q)}}  \sum_{\alpha \in Q }  q^{\frac{1}{2} || \sqrt{p}  \left(\alpha+\rho+\hat{\lambda}\right)+\bar{\lambda} -\frac{1}{\sqrt{p}} \rho ||^2} \left( \sum_{w \in W } (-1)^{l(w)} \frac{{\bf z}^{w\left(\alpha+\rho+\hat{\lambda}\right)-\rho}}{\delta({ z})} \right),
\]
where $\gamma'$ is as above.
After we expand the Weyl denominator 
in terms of Kostant's partition function and extracting the constant term, we easily get 
\begin{equation}
{\rm ch}[W^0(p,\mu)_Q](q)=\frac{(-1)^{|\Delta_+|}}{\eta(\tau)^{rank(Q)}} \sum_{\beta \in Q_+} \sum_{w \in W} (-1)^{\ell(w)} q^{\frac{1}{2}||\sqrt{p}w^{-1}(\beta - \hat{\lambda}-\mu )+
(-\overline{\lambda}+\frac{\rho}{\sqrt{p}})||^2}K(\beta),
%\bar{\lambda}-\frac{\rho}{\sqrt{p}}||^2} K(\beta), 
\end{equation}
where $Q^+$ is the cone $k_1 \alpha_1+\cdots + k_n \alpha_n$; $k_i \in \mathbb{Z}_{\geq 0}$.
By using the invariance property of the Weyl group element $w$, this expression can be also written as 
\begin{align*}
& \frac{(-1)^{|\Delta_+|}}{\eta(\tau)^{rank(Q)}} 
\sum_{\beta \in Q_+} \sum_{w \in W} (-1)^{\ell(w)} q^{\frac{1}{2}||\sqrt{p}(\beta -\hat{\lambda}-\mu)+w(- \bar{\lambda}+\frac{\rho}{\sqrt{p}})||^2} K(\beta) =
  \frac{(-1)^{|\Delta_+|}}{\eta(\tau)^{rank(Q)}}  \sum_{w \in W} (-1)^{\ell(w)} A_w(\tau),
\end{align*}
where we let 
\begin{align*}
 A_w(\tau)&=\sum_{\beta \in Q_+} K(\beta) q^{\frac{1}{2}|| \sqrt{p}(\beta -\hat{\lambda}-\mu)+w(- \bar{\lambda}+\frac{\rho}{\sqrt{p}}) ||^2}  \\
& =  \sum_{\beta \in Q_+ +\frac{1}{2}\sum_{i=1}^n \alpha_i} K \left(\beta - \frac{1}{2} \sum_{i=1}^n \alpha_i\right)  q^{\frac{1}{2}|| \sqrt{p}(\beta -\hat{\lambda}-\mu-\frac{1}{2} \sum_i \alpha_i)+w(- \bar{\lambda}+\frac{\rho}{\sqrt{p}}) ||^2}
%q^{a(\mu,\lambda,w)} \sum_{\beta \in Q^+ +\frac{1}{2}\sum_{i=1}^n \alpha_i} K \left(\beta - \frac{1}{2} \sum_{i=1}^n \alpha_i\right) q^{\frac{1}{2}p (\beta,\beta)-(\beta,\frac{p}{2} \alpha_i + p \mu%+w(\sqrt{p}\bar{\lambda}-\rho))}.
\end{align*}
%Here we are using
%$$a(\mu,\lambda,w)=\frac{1}{2} \left(\frac{p}{2} \alpha_i + p \mu+w(\sqrt{p}\bar{\lambda}-\rho), \frac{p}{2} \alpha_i + p \mu+w(\sqrt{p}\bar{\lambda}-\rho)\right)   .$$
Using that $Q_+ = \mathbb Z_{\geq 0}\alpha_1 \oplus \dots \oplus \mathbb Z_{\geq 0}\alpha_n$ with quadratic form given by the Gram matrix $G$ of the root lattice $Q$, we get
$$A_w(\tau)= q^{\frac{1}{2}\left(v, Av\right)} \sum_{k \in (\mathbb{Z}_{\geq 0}+ \frac{1}{2})^n} K \left((k_1-\frac{1}{2})\alpha_1+ \cdots +(k_n-\frac{1}{2})\alpha_n \right)q^{\frac{1}{2}(Ak,k)+(Ak,v)},$$
where $A=pX$, where $X$ is the corresponding ADE Gram matrix, and hence
$$v=-\left(\frac{1}{2},...,\frac{1}{2}\right) -\mu - \hat{\lambda}+w\left(-\frac{\bar{\lambda}}{\sqrt{p}}+\frac{\rho}{p}\right).$$
With  $K_\epsilon(u, \tau)$ as in Section 2.1, we get 
$$A_w(\tau)= q^{\frac{1}{2}\left(v, Av\right)}   K_{\epsilon=0}(v \tau,\tau).$$

\subsection{Regularized characters}

As in \cite{CM1}, typical modules can be easily regularized in a canonical way 
We simply let 
\begin{equation} \label{char-generic-reg} 
{\rm ch}[F_{\lambda}]^\epsilon(\tau)=\frac{e^{2 \pi (\epsilon, \lambda- (\sqrt{p} - \frac{1}{\sqrt{p}}) \rho )}q^{||\lambda-( \sqrt{p}-\frac{1}{\sqrt{p}} ) \rho||^2/2}}{\eta(\tau)^n}.
\end{equation}
We thus define the regularized character of atypical module
\begin{equation}\nonumber
\begin{split}
{\rm ch}[W^0(p,\mu)_Q]^\epsilon(\tau) &:= \frac{(-1)^{|\Delta_+|}}{\eta(\tau)^{rank(Q)}} \sum_{w \in W} (-1)^{\ell(w)} A_w^\epsilon(\tau) \\
A_w^\epsilon(\tau)&:= e^{2\pi(v+\frac{1}{2}(1, \dots, 1), \epsilon)} q^{\frac{1}{2}\left(v, Av\right)}   K_{\epsilon}(v \tau,\tau).
\end{split}
\end{equation}
%{\bf TC: changed the regularization for atypiclas here to make results agree with sl2 case}
\begin{remark} {\em We note that the above $\epsilon$-regularization slightly differs from the original used  in \cite{CM1} for $Q=A_1$.
The expression in (\ref{char-generic-reg}) is actually   ${\rm ch}[F^{\frac{\epsilon}{\sqrt{2}}}_{\lambda}](\tau)$ if we use formula 
(3.4) in \cite[Section 3.3]{CM1}.
}
\end{remark}

\section{Modular transformations and regularized quantum dimensions of $W^0(p)_Q$-modules}

In this part, as an application of results from Chapter 2 and 3, we first compute modular transformation of regularized atypical and typical characters and 
the corresponding regularized quantum dimensions.

\begin{proposition} Let $\alpha_0=\sqrt{p}-\sqrt{1/p}$.
We have
$${\rm ch}[F_{\lambda+ \alpha_0 \rho}]^\epsilon \left(-\frac{1}{\tau} \right)=\int_{\mathbb{R}^n} S^\epsilon_{\lambda+\alpha_0 \rho, \mu+\alpha_0 \rho} {\rm ch}[F_{\mu}]^\epsilon(\tau) d \mu,$$
where the  $S$-kernel  is $S^\epsilon_{\lambda+\alpha_0 \rho, \mu+\alpha_0 \rho}=e^{2 \pi ( \epsilon, \lambda - \mu)} e^{ -2 \pi i (\lambda,A^{-1} \mu)}$.
\end{proposition}
\begin{proof}
The relevant Gaussian integral was essentially computed in the first formula in the proof of Theorem \ref{first}. The rest is just adjustment of parameters including 
a shift by $\alpha_0 \rho$.
\end{proof}

\subsection{Regularized quantum dimensions for ${\rm Re}(\epsilon_i)<0$}

An application of Lemma \ref{negative}, with $u=-v$, immediately gives 
\begin{corollary} \label{gauss-char}
Let ${\rm Re}(\epsilon_i)<0$ for $i=1, \dots, n$, then 
\begin{align*}
{\rm ch}[W^0(p,\mu)_Q]^\epsilon \left(-\frac{1}{\tau}\right) & =  \\
& \frac{(2i)^{-|\Delta_+|}}{ \eta(\tau)^{rank(Q)} \sqrt{{\rm det}(A)}} \int_{\mathbb R^n} \frac{q^{\frac{1}{2}\left(w, A^{-1}w\right)} e^{2\pi i (\hat{\lambda}+\mu, w+i\epsilon)} {\rm num}_{(- \sqrt{p} \bar{\lambda})}(-\frac{w+i\epsilon}{p})}{\prod\limits_{\alpha \in \Delta^+} 
\sin(\alpha, w + i \epsilon)}d^nw,
\end{align*}
where ${\rm num}_\lambda(x)=\sum\limits_{w \in W } (-1)^{\ell(w)} e^{2\pi i\left(w(\lambda+\rho), x\right)}$ is the Weyl numerator of the irreducible highest-weight module $V_\lambda$ of highest-weight $\lambda$ of the simple finite-dimensional Lie algebra $\frak{g}$. 
\end{corollary}
\begin{proof}
Follows immediately from the definition of ${\rm ch}[W^0(p,\mu)_Q]^\epsilon$ and Lemma \ref{negative}.

\end{proof}
For a $W^0(p)_Q$-module $M$ we define its regularized quantum dimension
$${\rm qdim}[M]^\epsilon = \lim_{\tau \rightarrow 0+} \frac{{\rm ch}[M]^\epsilon(\tau)}{{\rm ch}[W^0(p,0)_Q]^\epsilon(\tau)},$$
where the limit is taken along the imaginary axis.

\begin{corollary} We have
\begin{itemize}
\item (Atypical)
\begin{equation}\nonumber
{\rm qdim}[W^0(p,\mu)_Q]^\epsilon = e^{-2\pi (\gamma+\hat{\lambda}, \epsilon)} \chi_{- \sqrt{p} \bar{\lambda}}\left(\frac{-i\epsilon}{p}\right),
\end{equation}
where $\chi_{\lambda}$ denotes the Weyl character of $V(\lambda)$ and 
${\rm qdim}[M^\epsilon]$ denoted the regularized quantum dimension as in \cite{CM1}.
\item (Typical) 

%{\bf There seems to be an extra term. The extra term was due to an inconsistent regularization that I changed}
\begin{equation}\nonumber
{\rm qdim}[F_\lambda]^\epsilon = e^{2\pi (\epsilon, \lambda - \alpha_0 \rho)}  \prod_{\alpha \in \Delta^+} \frac{\sin \left((\alpha, i \epsilon/p)\right)}{\sin \left((\alpha, i \epsilon)\right)}.
\end{equation}
\end{itemize}

\end{corollary}
\begin{proof}

We shall give two different proofs. The first proof is given as in \cite{CMW}, by using asymptotic properties of the generalized Gauss' integral applied to
${\rm ch}[W^0(p,\mu)_Q]^\epsilon\left(-\frac{1}{\tau}\right)$ as expressed in Corollary \ref{gauss-char}. Notice also that for $W^0(p,0)_Q$ we have $\hat{\lambda}=\overline{\lambda}=\gamma=0$. Then we get 
\begin{equation}
\begin{split}
{\rm qdim}[W^0(p,\mu)_Q]^\epsilon &= \lim_{\tau \rightarrow 0} \frac{{\rm ch}[W^0(p,\mu)_Q]^\epsilon(\tau)}{{\rm ch}[W^0(p,0)_Q]^\epsilon(\tau)} 
= \lim_{\tau \rightarrow i\infty} \frac{{\rm ch}[W^0(p,\mu)_Q]^\epsilon\left(-\frac{1}{\tau}\right)}{{\rm ch}[W^0(p,0)_Q]^\epsilon\left(-\frac{1}{\tau}\right)} \\
&= e^{-2\pi (\gamma + \hat{\lambda}, \epsilon)} \frac{{\rm num}_{-\sqrt{p}\overline{\lambda}}\left(\frac{- i \epsilon}{p} \right)}{{\rm num}_{0}\left(\frac{-i \epsilon}{p}\right)}=e^{-2\pi (\gamma+\hat{\lambda}, \epsilon)} \chi_{- \sqrt{p} \bar{\lambda}}\left(\frac{-i\epsilon}{p}\right),
\end{split}
\end{equation}
where it the last line we used the Weyl character (and Weyl denominator) formula.
Even easier proof is obtained directly from the definition of $ \frac{(-1)^{|\Delta_+|}}{\eta(\tau)^{rank(Q)}} \sum_{w \in W} (-1)^{\ell(w)} A_w^\epsilon(\tau)$
by noticing that for  $A_w^\epsilon(\tau)|_{\tau=0}$ is convergent in the given $\epsilon$-region. After we quotient two expressions we get the same formula.

\end{proof}

\subsection{Regularized quantum dimensions for ${\rm Re}(\epsilon_i)>>0$}

%{\bf AM: Thomas, I fixed this part - but I think we need another condition - see the theorem. I agree with that second condition I overlooked the boundary points}

 Let us now restrict to  the region ${\rm Re}(\epsilon_i)>0$ for all $i$. Then if all $Re(\epsilon_i)$ are sufficiently large the term \eqref{eq:Y} dominates the asymptotic behaviour of regularized characters. The reason is as follows. First, all entries of the Gram matrix $A^{-1}$ of the rescaled weight lattice are positive integers. 
 Second, let $\epsilon_i = x_i + i y_i$ and let $k=(k_1,...,k_n) \in J(y)$ with the fundamental cell associated to $y$   defined as
 \[
 J(y) := \left\{ \ m \in  \mathbb Z^n\ | \ |m_i+y_i | < \frac{1}{2}\ \ \text{for} \ i=1, \dots, n \ \right\}.
 \]
 Here $m = \frac{1}{\sqrt{p}}\sum_i m_i \omega_i$ and similarly for $y$. 
 Then the dominating term of $Y_\epsilon$ in the limit $i \tau\rightarrow \infty$ behaves as $q^d$ (up to a rational function in $\sqrt{\tau}$) with 
 \begin{equation} \label{d-quant}
 d=-\frac{1}{2} \left({\rm Re}((\epsilon+ik), A^{-1}(\epsilon+ik))\right)
 \end{equation}  while the term corresponding to 
 %{\bf AM: same comment as on p.9} 
 \[
 \prod_{j=1}^\ell D_{\epsilon_{r_j}}^{m_{r_j}} \int_{\mathbb R^{n-\ell}} 
\frac{\theta^{\{r_1, \dots, r_\ell\}}\left(q^{\frac{1}{2}\left(w, A^{-1}w\right)}e^{2\pi i (u, w)} \prod\limits_{k=1}^\ell e^{\pi i\left(w_{r_k}+m_{r_k}i\epsilon_{r_k}\right)} \right)}{\Delta^{(r_1, m_{r_1}), \dots, (r_\ell, m_{r_\ell})}(w+i\epsilon)} dw^{n-\ell}_{r_1,\dots, r_\ell}
 \] 
behaves as $q^e$ (again up to a rational function in $\sqrt{\tau}$) with $e= \frac{1}{2} \left({\rm Re} \left((\epsilon+ik), B(\epsilon+ik) \right)\right)$.
Here $B$ is the matrix obtained from $A^{-1}$ by letting all entries to be zero that have at least one index not appearing in the set $\{ r_1, \dots, r_\ell\}$. Since all entries of $A^{-1}$ are positive it is guaranteed that $d>e$ for any choice of $\{r_1, \dots, r_\ell\} \neq \{1, \dots, n\}$ if (recall Re$(\epsilon_i)=x_i$)
\begin{equation}\label{eq:large}
x_i \left(A^{-1}\right)_{ii}x_i > \frac{1}{4} \sum_{\ell,j=1}^n \left(A^{-1}\right)_{\ell j}
\end{equation}
for all $i=1, \dots, n$.
The leading coefficient of ${\rm ch}[W^0(p,0)_Q]^\epsilon(-1/\tau)$ in that region is then proportional to $\Delta(e^{-2\pi i k/p})$. We have to ensure that this coefficient is non-zero, hence define
\[
N = \left\{ \ k \in \mathbb Z^n \ \Big\vert\ \Delta(e^{-2\pi i k/p})\neq 0\ \right\}.
\]
Putting now things together it follows that
\begin{theorem} \label{quantum-dim-pos} Suppose that (i)  $J(y) \cap N \neq \emptyset$ and all $x_i$ are sufficiently large in the sense of equation \eqref{eq:large}, and (ii) there is a unique 
$k \in  \mathbb{Z}^n$ minimizing the quantity $d$ in (\ref{d-quant}). Then 
\begin{itemize} 
\item (Atypical)
\begin{equation} \label{qdim-discrete}
\begin{split}
{\rm qdim}[W^0(p,\mu)_Q]^\epsilon = \lim_{\tau \rightarrow i\infty} \frac{{\rm ch}[W^0(p,\mu)_Q]^\epsilon\left(-\frac{1}{\tau}\right)}{{\rm ch}[W^0(p,0)_Q]^\epsilon\left(-\frac{1}{\tau}\right)} = e^{-2\pi (\gamma+\hat{\lambda}, k)} \chi_{-\sqrt{p}\bar\lambda}\left(-\frac{k}{p}\right).
\end{split}
\end{equation}
\item (Typical) For any $\lambda \in \frak{h}^*$, 
$${\rm qdim}[F_\lambda]^\epsilon = 0.$$
%provided $J(y) \cap N \neq \{ \ \ \}$ and all $x_i$ are sufficiently large in the sense of equation \eqref{eq:large}. 
\end{itemize}
\end{theorem}
\begin{proof} 
We only have to discuss the typicals. But this is clear due to dominating term in the denominator coming from ${\rm qdim}[W^0(p,0)_Q]^\epsilon$.

\end{proof}

\section{Conclusion}

We end here with a few of comments for future work and a remark.

\begin{enumerate}[labelsep=1em]
%\item It is desirable to extend results in Section 3 to types $D$ and $E$. For the $D$-type this can be done in more-or-less straightforward manner (for the sake of brevity we omit it in this paper), but the $E$-series is certainly more difficult.

\item Presently, it seems difficult to completely describe behavior of ${\rm qdim}[W^0(p,\mu)_Q]^\epsilon$ outside the ${\rm Re}(\epsilon_i)<0$ orthant, even if the root 
lattice is of type $A$. In fact, already in the rank one case \cite{CMW} we noticed that inside ${\rm Re}(\epsilon)>0$ nontrivial (sub)regions have to handled with care, including Stokes lines, imaginary axis, "fuzzy" lines. These special one dimensional (over reals) domains were later worked out in \cite{BFM}. We hope to return to this problem in our future publications, at least for the $A_2$ root lattice for all values of $\epsilon_1$ and $\epsilon_2$.

\item Note that our formula for regularized quantum dimensions (\ref{qdim-discrete}) is closely related to quantum dimension formulas for the WZW model of $ADE$-type.
More generally, it is known that for an affine Kac-Moody Lie algebra of type $X^{(1)}_n$, the formula for the quantum dimension of $L(\ell \Lambda_0)$-module
$L(\ell, \lambda)$ is given by ${\rm qdim}(L(\ell,\lambda))=\chi_{\lambda}\left(- \frac{(\lambda,\rho)}{\ell + h^\vee}\right)$. Here $\ell \in \mathbb{N}$ is the level.
We expect that, as in the $\frak{sl}_2$ case corresponding to the $(1,p)$-singlet algebra \cite{CMW}, regularized quantum dimensions in a certain infinite subregion in ${\rm Re}(\epsilon_i)>0$ 
capture correctly the fusion ring of the corresponding WZW models at level $p-h^\vee$. Notice that the condition $p \geq h^\vee$ seems to be important to make a link with affine Lie algebras. This condition might be related to non-vanishing of the regularized quantum dimensions. 
\end{enumerate}

\begin{remark}
We finish with a comment that for $\epsilon=0$ (the usual quantum dimension) computation has to be done differently as it corresponds to a point on the product of imaginary axes that we removed at the very start. But this case was already proven in \cite{BM2}:
$$ {\rm qdim}[W^0(p,\mu)_Q]={\rm dim}_{\mathbb{C}} V(-\sqrt{p} \overline{\lambda}),$$
while for typicals
$$ {\rm qdim}[F_\lambda]=p^{|\Delta_+|}.$$
This nicely agrees with our results for the ${\rm Re}(\epsilon_i)<0$ region (notice that the left limit of quantum dimensions in ${\rm Re}(\epsilon_i)<0$  gives the above values). 
 \end{remark}

\hspace*{1cm}

\noindent Department of Mathematical and Statistical Sciences, University of Alberta,
Edmonton, Alberta  T6G 2G1, Canada. 
\emph{email: creutzig@ualberta.ca}

\hspace*{5mm}

\noindent{{\em Current address:} Max Planck Institute f\"ur Mathematik, Vivatsgasse 7, Bonn, Germany.}

\noindent {{\em Permanent address:} Department of Mathematics and Statistics, SUNY-Albany, 1400 Washington Avenue, Albany, NY 12222, USA.}
\emph{email: amilas@albany.edu}


\begin{thebibliography}{FGST2}

\bibitem{AB} G. Andrews and B. Berndt, Ramanujan's Lost Notebook: Part II, Springer, 2009.

\bibitem{ACR} J. Auger, T. Creutzig and D, Ridout, Modularity of logarithmic parafermion vertex algebras, in preparation.

\bibitem{AG}  G. E. Andrews and F. G. Garvan, Dyson's crank of a partition. Bull. Amer. Math. Soc. (N.S.) 18 (1988), no. 2, 167-171.

\bibitem{AM1} D. Adamovic and A. Milas, Logarithmic intertwining operators and $W(2,2p-1)$-algebras, Journal of Mathematical Physics, 48 073503 (2007).

\bibitem{AM4}  D. Adamovic and A. Milas, On the triplet vertex algebra $W(p)$,  Advances in Mathematics 217  (2008), 2664-2699.

\bibitem{AM} D. Adamovic and A. Milas,   $C_2$-cofinite vertex algebras and their logarithmic modules, Proceedings of the conference "Conformal Field Theories and Tensor Categories, Beijing Mathematical Lectures from Beijing University, Vol. 2, ( 2014), 18 pp.

\bibitem{AM3} D.Adamovic and A. Milas, Some applications and constructions of intertwining operators in LCFT, submitted. {\tt  arXiv:1605.05561}.

\bibitem{Ar1} T. Arakawa, Representation Theory of W-Algebras, Invent. Math., Vol. 169 (2007), no. 2, 219--320.

\bibitem{Ar2} T. Arakawa, Two-sided BGG resolutions of admissible representations, Represent. Theory 18 (2014), 183-222.

\bibitem{AW} G. Andrews and S. O. Warnaar, The product of partial theta functions, Advances in Applied Mathematics 39.1 (2007): 116-120.

\bibitem{BFN} A.  Braverman, M. Finkelberg, and H. Nakajima, Instanton moduli spaces and $W$-algebras, preprint. arXiv:1406.2381.

\bibitem{BCR} K. Bringmann, T. Creutzig and L. Rolen, Negative index Jacobi forms and quantum modular forms. Res. Math. Sci. 1 (2014), Art. 11, 32 pp.

\bibitem{BFM} K. Bringmann, A. Folsom, and A. Milas,   Asymptotic behavior of partial and false theta functions arising from Jacobi forms and regularized characters, submitted; {\tt  arXiv.1604.01977}.

\bibitem{BRZ} K. Bringmann, L. Rolen and S. Zwegers, On the Fourier coefficients of negative index meromorphic Jacobi forms,  arXiv:1501.04476.

\bibitem{BM1} K. Bringmann and A.Milas, W-Algebras, False Theta Functions and Quantum Modular Forms, I,  { Int. Math. Res. Notices} (2015), 11351-11387.

\bibitem{BM2} K. Bringmann and A.Milas, W-Algebras, Higher Rank False Theta Functions and Quantum Dimensions, submitted.

\bibitem{CG} T. Creutzig and T. Gannon, \textit{The Theory of $C_2$-cofinite VOAs}, in preparation.

\bibitem{CG2} T. Creutzig and T. Gannon, \textit{Logarithmic conformal field theory, log-modular tensor categories and modular forms}, arXiv:1605.04630.

\bibitem{CM1} T. Creutzig and A. Milas, False Theta Functions and the Verlinde formula, Advances in Mathematics, 262 (2014), 520-554.

\bibitem{CMR} T. Creutzig. A. Milas and M. Rupert, Logarithmic Link Invariants of $\overline{U}^H_q(sl_2)$ and Asymptotic Dimensions of Singlet Vertex Algebras, submitted {\tt  arXiv:1605.05634}.

\bibitem{CMW} T. Creutzig, A. Milas and S. Wood, On regularized quantum dimensions of the singlet vertex operator algebra and false theta functions, published in IMRN, arXiv:1411.3282.

\bibitem{CRW} T. Creutzig, D. Ridout and S. Wood, Coset Constructions of Logarithmic (1,p)- Models, Lett. Math. Phys. 104, 5 (2014) 553--583.

\bibitem{FGST} B. L. Feigin, Gainutdinov, A. M., Semikhatov, A. M.,  Tipunin, I. Y., Modular group representations and fusion in logarithmic conformal field theories and in the quantum group center. Communications in mathematical physics, 265 (2006), 47-93.

\bibitem{FT} B. Feigin and  I. Tipunin, Logarithmic CFTs connected with simple Lie algebras, preprint; arXiv:1002.5047.

\bibitem{FKR} A. Folsom,  W. Kohnen and S. Robins, Conic theta functions and their relations to theta functions,  Annales de l'Institut Fourier (Grenoble), 65 no. 3 (2015), 1133-1151.


\bibitem{FrB}  E. Frenkel and D. Ben-Zvi, {\em Vertex algebras and algebraic curves},
Mathematical Surveys and Monographs, 88, American Mathematical
Society, Providence, RI, 2001.

\bibitem{GK} M. Gorelik and V. Kac, On simplicity of vacuum modules, {\em Advances in Mathematics}, Issue 2, 1 (2007), 621-677.

\bibitem{GL}  S. Garoufalidis and T. Q. L\^e, Nahm sums, stability and the colored Jones polynomial. Res. Math. Sci. 2 (2015), Art. 1, 55 pp.

\bibitem{H} Y.-Z. Huang, {Rigidity and modularity of vertex tensor categories}, Commun. Contemp. Math. 10 (2008) 871--911. 


\bibitem{M1} A.Milas, Characters of modules of irrational vertex algebras, {\em Proceedings of the Conference on Vertex Algebras and Automorphic Forms},  Heidelberg, 2011 {\bf 8} MATCH, Springer, (2014).

\bibitem{MP} A. Milas and M. Penn, Lattice vertex algebras and combinatorial bases: general case and $W$-algebras, New York J. Math 18 (2012): 621-650.

\bibitem{Mi} M. Miyamoto, Modular invariance of vertex operator algebras satisfying $C_2$-cofiniteness. Duke Math. J. 122 (2004), no. 1, 51-91.

\bibitem{S} A.M.  Semikhatov, A note on the" logarithmic-$W_3$" octuplet algebra and its Nichols algebra, arXiv:1301.2227.

\bibitem{Wa} S. O. Warnaar, Partial theta functions. I. Beyond the lost notebook, Proceedings of the London Mathematical Society 87.2 (2003): 363-395.

%\bibitem{Z} D. Zagier, Quantum modular forms, Quanta of maths 11 (2010), 659-675.


\end{thebibliography}
\end{document}